\documentclass[graybox]{svmult}

\usepackage{mathptmx}       
\usepackage{helvet}         
\usepackage{courier}        
\usepackage{type1cm}        
\usepackage{makeidx}         
\usepackage{graphicx}        
\usepackage{multicol}        
\usepackage[bottom]{footmisc}
\usepackage[tbtags]{amsmath}
\usepackage{amsfonts,amssymb,mathrsfs,amscd}

\usepackage{graphicx}

\usepackage[matrix,arrow,curve]{xy}

\makeindex             

\begin{document}

\title*{Identity theorem for pro-$p$-groups.}
\titlerunning{Identity theorem for pro-$p$-groups}
\author{A.\,M.~Mikhovich}

\institute{Lomonosov Moscow State University, \email{mikhandr@mail.ru}}

\maketitle

\abstract{The concept of schematization consists in replacing simplicial groups by simplicial affine group schemes. In the case when the coefficient field has a zero characteristic, there is a prominent theory of simplicial prounipotent groups,
the origins of which lead to the rational homotopy theory of D. Quillen. The specificity of a prime finite field $ \mathbb{F}_p $ is that the Zariski topology on the $n$-dimensional affine space $\mathbb{F}_p^n $ turns out to be discrete. Therefore, although finite $ p $-groups are $ \mathbb{F}_p $-points of $ \mathbb{F}_p $-unipotent affine group schemes, this observation is not actually used. Nevertheless, schematization reveals the profound properties of $\mathbb{F}_p$-prounipotent groups, especially in connection with prounipotent groups in the zero characteristic and in the study of quasirationality. In this paper, using results on representations and cohomology of prounipotent groups in characteristic 0, we prove an analogue of Lyndon Identity theorem for one-relator pro-p-groups (question posed by J.P. Serre) and demonstrate the application to one more problem of J.-P. Serre concerning one-relator pro-$p$-groups of cohomological dimension 2.
Schematic approach makes it possible to consider the problems of pro-$ p $-groups theory through the prism of Tannaka duality, concentrating on the category of representations. In particular we attach special importance to the existence of identities in free pro-$p$-groups (``conjurings'').}

\section{Introduction} \label{s1}
Part (1.1) of the introduction contains a brief review of a modern paradigm (as it seen by the author) of the pro-$p$-group theory and an explanation of the importance of one-relator pro-$p$-groups. In (1.2) we remind basic definitions of prounipotent group theory, those we need in the sequel. Section (1.3) provides condensed introduction to the results on quasirational presentations and their schematization. We also include the proof of quasirationality for presentations of one-relator pro-$p$-groups. Section (1.4) will help to understand the motivations for our main result (Theorem \ref{t02}), elucidating why one should consider it as an analog of Lyndon Identity theorem. We explain the importance of Tannakian philosophy in (1.5) presenting the construction of ``conjurings''.

\subsection{Pro-$p$-groups with a single defining relation}\label{s1.1}
By a pro-$p$-group one calls a group isomorphic to a projective limit of finite $p$-groups.
This is a topological group (with the topology of projective limit)
which is compact and totally disconnected.
For such groups one has a presentation theory similar in many aspects to the combinatorial
theory of discrete groups \cite{Koch}, \cite{ZR}.

Let us say that a pro-$p$-group $G$ is defined by a finite type pro-$p$-presentation if $G$
is included into an exact sequence
\begin{equation}
1 \rightarrow R\rightarrow F \xrightarrow{\pi} G \rightarrow 1 \label{eq1}
\end{equation}
in which $F=F(X)$ is the free pro-$p$-group with a finite
set $X$ of generators, and $R$ is a closed normal subgroup topologically generated by a finite set $Y$
of elements in $F$, contained in the Frattini subgroup of $F$ (\cite{Koch}, \cite{ZR}).

If $A$ is a (profinite) ring, then denote by $AG$ the (completed) group algebra of the (pro-$p$) group $G$.
By the completed group algebra we understand the topological algebra $AG=\varprojlim AG_{\mu}$ \cite{ZR},
where $A=\varprojlim A_{\alpha}$ is a profinite ring ($A_{\alpha}$ are finite rings), and
$G=\varprojlim G_{\mu}$ is a decomposition of the pro-$p$-group $G$ into a projective limit of finite
$p$-groups $G_{\mu}$.

The interest to pro-$p$-groups in the recent years is related, first of all, to problems
which arose in a joint area of noncommutative geometry, topology, analysis, and group theory.
They play an important role in papers on the problems of Kadison--Kaplansky \cite{BV},
Atiyah \cite{LS}, and Baum--Connes \cite{S}.
Let us mention the concept of \emph{cohomological p-completeness} and a program (following from these problems) of studing pro-$p$-groups
whose discrete and continuous cohomologies with coefficients in $\mathbb{F}_p$ are isomorphic \cite[Chapter 5]{AF},\cite{FerKazRem07}.

Complete group rings of pro-$p$-groups are complete $\mathbb{F}_p$--Hopf algebras \cite{Qui5},
and pro-$p$-groups themselves are analogs of Malcev groups in positive characteristics,
hence the theory of presentations of pro-$p$-groups can be considered as two-dimensional $p$-adic
homotopy theory. By $p$-adic homotopy theory we mean the analog of Quillen's
rational homotopy theory in positive characteristics.
Despite the papers published already (for instance, \cite{Man,Lur}), such theory
remains mostly conjectural,
hence the potential of combinatorial pro-$p$-group theory is firstly in that we can
check rightness of new concepts in application to solving open problems (in particular listed above).

The origins of cohomological and combinatorial theory of pro-$p$-groups lie in the early
papers of J.-P. Serre and J. Tate, and they took the modern form in the monograph \cite{Se4}.
Some important results in discrete group theory arose as analogs of similar statements
on pro-$p$- groups \cite{Lub2}. For example, the celebrated Stallings theorem,
stating that a discrete group is free if and only if its cohomological dimension
equals one, arose from the analogy, proposed by J.-P. Serre,
with the known result from pro-$p$-group theory \cite[Corollary 2, p. 30]{Se4}.
Nevertheless, after first bright successes of the theory such as the Shafarevich
theorem on existence of infinite tower of class fields \cite[I.4.4, Theorem 2]{Se4}
(proof uses the Golod--Shafarevich inequality \cite[I.4.4, Theorem 1]{Se4})
and the Demushkin--Labute classification of Galois groups of maximal $p$-extensions of
$p$-adic fields in terms of generators and relations \cite[II.5.6]{Se4}, \cite{Se},
it became clear that the study of pro-$p$-groups given combinatorially sometimes
leads to more complicated structures than in the case of similar discrete presentations.
Thus, the first nontrivial question of J.-P. Serre on the structure of relation modules
of pro-$p$-groups with one relation, stated at the Bourbaki seminar 1962/63 \cite[10.2]{Se},
still waits for a final answer (we expect a counterexample).

Understanding pro-$p$-groups with one relation $r\in R\subseteq F^p[F,F]$ plays an essential role.
Actually, suppose $G_r=F/(r)_F$ be a pro-$p$-group with one relation (in the notations \eqref{eq1})
without torsion but with cohomological dimension greater than two.
Then $\overline{\frac{\partial r}{\partial x_i}}\in \mathbb{Z}_pG$, the images of the Fox partial derivatives
$\frac{\partial r}{\partial x_i}\in\mathbb{Z}_pF$ with respect to the homomorphism of completed group
rings $\mathbb{Z}_pF\twoheadrightarrow\mathbb{Z}_pG$ induced by the homomorphism $\pi$ from \eqref{eq1},
are divisors of 0 in the completed group ring $\mathbb{Z}_pG$ of the torsion free pro-$p$-group.
To see this, consider the Crowell--Lyndon sequence \cite[Theorem 2.2]{I} of $\mathbb{Z}_pG$- modules takes the form
$$0\rightarrow \pi_2\rightarrow\mathbb{Z}_pG\xrightarrow{\psi} \mathbb{Z}_pG^{|X|}\rightarrow IG\rightarrow 0,$$
where $IG$ is the augmentation ideal in $\mathbb{Z}_pG,$ and $\pi_2=ker\psi$. Here $\psi$ is defined by the rule
$$\psi(\alpha)=(\alpha\overline{\frac{\partial r}{\partial x_1}},..,\alpha\overline{\frac{\partial r}{\partial x_i}},..),
i=1..\mid X \mid=dim_{\mathbb{F}_p}H^1(G,\mathbb{F}_p).$$
By the Koch theorem \cite[Proposition 7.7]{Koch}, cohomological dimension of a pro-$p$- group $G$
equals 2 if and only if $\pi_2=0$. Hence the assumption $cd(G)>2$ is equivalent to the statement that
for all $i\in I,$ where $\mid I \mid=dim_{\mathbb{F}_p}H^1(G,\mathbb{F}_p)$, the images
of Fox partial derivatives $\overline{\frac{\partial r}{\partial x_i}}\in \mathbb{Z}_pG$
are divisors of 0 of nontrivial elements in $\mathbb{Z}_pG$.
We must note that $\mathbb{Z}_pG$ is not in general embedded into the group von Neumann algebra
$\mathcal{N}(G)$ 
and hence one needs additional arguments in favor of the statement that such conjectural pro-$p$-groups
are related to the problems listed above.
\newpage

\subsection{Prounipotent groups} \label{s1.2}
By an affine group scheme over a field $k$ one calls a representable $k$-functor $G$
from the category $Alg_k$ of commutative $k$-algebras with unit to the category of groups.
If $G$ is representable by the algebra $\mathcal{O}(G)$, then as a functor $G$ is given, for any commutative
$k$-algebra $A$, by the formula
$$G(A) = Hom_{Alg_k}(\mathcal{O}(G),A).$$
Of course, we assume that the considered homomorphisms $Hom_{Alg_k}$ take the unit of the algebra $\mathcal{O}(G)$
to the unit of the $k$-algebra $A$. The algebra $\mathcal{O}(G)$ representing the functor $G$
is usually called the \emph{algebra of regular functions} of $G$.
The Yoneda lemma implies the anti-equivalence of the categories of affine group schemes and commutative
Hopf algebras \cite[1.3]{Wat}.
Let us say that an affine group scheme $G$ is \emph{algebraic} if its Hopf algebra of regular functions $\mathcal{O}(G)$
is finitely generated as the commutative $k$-algebra.


\begin{definition} \label{d5} By a unipotent group one calls an affine algebraic group scheme $G$
whose Hopf algebra of regular functions $\mathcal{O}(G)$ is conilpotent
(or coconnected, for equivalent definitions see \cite[8.3]{Wat}, \cite[Proposition 16]{Vez}).
An affine group scheme $G=\varprojlim G_{\alpha},$ where $G_{\alpha}$ are affine algebraic group schemes
over a field $k$, is called a prounipotent group if each $G_{\alpha}$ is a unipotent group.
\end{definition}

There is also the well known correspondence between unipotent groups over a field $k$ of characteristics 0
and nilpotent Lie algebras over $k$, which assigns to a unipotent group its Lie algebra.
This correspondence is easily extended to the correspondence between prounipotent groups over $k$
and pronilpotent Lie algebras over $k$ \cite[Appendix A.3]{Qui5}.
Functoriality of the correspondence enables one, when it is convenient, to interpret
problems on unipotent groups in the language of Lie algebras. For example, the image of a closed subgroup
under a homomorphism of prounipotent groups will be always a closed subgroup.
The main theorems on the structure of normal series, nontriviality of the center of a unipotent group \cite[VII,17]{Ham}
are transferred from the corresponding statements for Lie algebras \cite[Part 1, Chapter V, \S 3]{Se2}.
By the Quillen theorem \cite[A.3, Theorem 3.3]{Qui5}, reconstruction of the algebra of regular functions
$\mathcal{O}(G)$ of a prounipotent group $G$ from the group of $k$-points $G(k)$
is made through the dual algebra by the formula
$\mathcal{O}(G)^*\cong\widehat{k}G(k),$
where $\mathcal{O}(G)^*=Hom_k(\mathcal{O}(G),k)$ and $\widehat{k}G(k)$ is the group algebra completed with respect to the augmentation ideal. Recall also that $G(k)\cong \mathcal{G}\mathcal{O}(G)^*$ \cite[Proposition 18]{Vez}, where $\mathcal{G}$ is the functor of group-like elements in CHA.
By the Campbell--Hausdorff formula \cite[Part 1, Chapter IV, \S 7]{Se2}, \cite[A.1]{Qui5} we have
$G(k)=exp\mathcal{PO}(G)^*,  \mathcal{PO}(G)^*=log\mathcal{G}\mathcal{O}(G)^*,$ where $\mathcal{P}$ is the functor of primitive elements.

Let $A$ be a Hopf algebra over a field $k$ of characteristics 0, in which: 1) the product is commutative;
2) the coproduct is conilpotent. Then, as an algebra, $A$ is isomorphic to a free commutative algebra
\cite[Theorem 3.9.1]{Car}. Thus, $A$ is the algebra of functions on an affine space,
and we can use results from the theory of linear algebraic groups in characteristics 0.

As in \cite[2]{HM2003}, let us call by the Zariski closure of a subset $S\subseteq G(k)$
the least affine subgroup $H$ in $G$ such that $S\subseteq H(k)$ is the projective limit
$\varprojlim H_{\alpha},$ where we have denoted by $H_{\alpha}$ the closure of the image of $S$ in $G_{\alpha}(k)$.

Pro-$p$-groups are $\mathbb{F}_p$-points of prounipotent affine group schemes defined over the field $\mathbb{F}_p$.
Indeed, consider the complete group algebra $\mathbb{F}_pG=\varprojlim \mathbb{F}_p[G_{\alpha}]$
of a pro-$p$-group $G=\varprojlim G_{\alpha},$ where $G_{\alpha}$ are finite $p$-groups.
Each group algebra $\mathbb{F}_p[G_{\alpha}]$ is obviously a cocommutative Hopf algebra over the field
$\mathbb{F}_p$. Then the dual Hopf algebra $\mathbb{F}_p[G_{\alpha}]^*$ \cite[3]{Mikh2016}
is a finitely generated commutative Hopf algebra, and therefore it defines certain affine
algebraic group scheme. Let $\mathcal{G}$ be the functor of group like elements of a Hopf algebra.
Note that \cite[Proposition 18]{Vez}
$$G_{\alpha}=\mathcal{G}\mathbb{F}_p[G_{\alpha}]\cong Hom_{Alg_{\mathbb{F}_p}}(\mathbb{F}_p[G_{\alpha}]^*,\mathbb{F}_p),$$
where $Hom_{Alg_{\mathbb{F}_p}}(\mathbb{F}_p[G_{\alpha}]^*,\_)$ is the functor from the category
of commutative $\mathbb{F}_p$-algebras with unit to the category of sets which assigns to each
commutative $\mathbb{F}_p$- algebra $A$ with unit the set $Hom_{Alg_{\mathbb{F}_p}}(\mathbb{F}_p[G_{\alpha}]^*,A)$
of homomorphisms $\phi:\mathbb{F}_p[G_{\alpha}]^*\rightarrow A$ of commutative $\mathbb{F}_p$-algebras with unit.
But $$G\cong\mathcal{G}\mathbb{F}_pG\cong \varprojlim\mathcal{G}\mathbb{F}_p[G_{\alpha}]\cong
\varprojlim Hom_{Alg_{\mathbb{F}_p}}(\mathbb{F}_p[G_{\alpha}]^*,\mathbb{F}_p)\cong
Hom_{Alg_{\mathbb{F}_p}}(\mathbb{F}_p[G]^*,\mathbb{F}_p).$$

It remains to note, that the Kolchin theorem \cite[Theorem 8.2]{Wat} (since the action of a finite $p$-group on a $\mathbb{F}_p$-vector space of a finite dimension always has a fixed point) implies unipotency.

\subsection{$QR$-presentations and their schematization} \label{s1.3}

For discrete groups, $p\geq2$ will run over all primes, and for pro-$p$-groups $p$ is fixed. Let $G$ be a
(pro-$p$)group with a finite type (pro-$p$)presentation \eqref{eq1},
$\overline{R}=R/[R,R]$ be the corresponding \emph{relation module}, where $[R,R]$ is the commutant,
and the action of $G$ is induced by conjugation of $F$ on $R$. For each prime number $p\geq2$
denote by ${\Delta}_p$ the augmentation ideal of the ring $\mathbb{F}_pG.$ In the pro-$p$-case,
by ${\Delta}^n$ we understand the closure of the module generated by $n$-th powers of elements
from $\Delta={\Delta}_p$, and in the discrete case it is the $n$-th power of the ideal ${\Delta}_p$
\cite{Pas}. The properties of this filtration in the pro-$p$-case are exposed in \cite[7.4]{Koch}; in the discrete
case, the properties of the Zassenhaus filtration are similar \cite[Ch. 11]{Pas}, the difference is
in the use of the usual group ring instead of the completed one.

Denote by $\mathcal{M}_n, n\in \mathbb{N}$ its Zassenhaus $p$-filtration in $F$ with coefficients in the field
$\mathbb{F}_p$, defined by the rule
$\mathcal{M}_{n,p}=\{f \in F\mid f-1 \in {\Delta}^n_p\}.$
We shall denote these filtrations simply by $\mathcal{M}_n$, omitting $p$, since its choice will be always clear
from the context. Let us introduce the notation $\mathbb{Z}_{(p)}$ for $\mathbb{Z}$ in the case
of discrete groups and for
$\mathbb{Z}_{p}$ in the case of pro-$p$-groups.

\begin{definition} We shall call presentation \eqref{eq1} quasirational ($QR$- presentation)
if one of the following three equivalent conditions holds:

\begin{description}[Type 1]
\item[(i)]{for each $n>0$ and for each prime $p\geq2$, the $F/R\mathcal{M}_n$-module $R/[R,R\mathcal{M}_n]$
has no $p$-torsion ($p$ is fixed for pro-$p$-groups and runs over all prime numbers $p\geq2$ and the corresponding
Zassenhaus $p$-filtrations in the discrete case).}
\item[(ii)]{the quotient module of coinvariants $\overline{R}_G=\overline{R}_F=R/[R,F]$ is torsion free.}
\item[(iii)]{$H_2(G,\mathbb{Z}_{(p)})$ is torsion free.}
\end{description}

\end{definition}

Proof of equivalence of conditions (i) - (iii) is contained in ~\cite[Proposition ~4]{Mikh2014} and
\cite[Proposition 1]{Mikh2017}. $QR$- presentations are curious in particular by the fact that they contain
aspherical presentations of discrete groups and their subpresentations, and also pro-$p$-presentations
of pro-$p$-groups with one relation. For the sake of completeness we present the proof of quasirationality for one-relator pro-$p$-groups. Let start with elementary lemma.

\begin{lemma} \label{l100} Let $G$ be a finite $p$-group, which acts on a finite abelian group $M$ of exponent $p$. Then the factor module of coinvariants $M_G=M/(g-1)M$ is nontrivial, where $(g-1)M$ - submodule of $M$, generated by elements of the form $(g-1)m$, $g \in G, m \in M$.
\end{lemma}

\begin{proof} We proceed by induction on rank $n$ of abelian group $M$. If $n=1,$ then $M=\mathbb{Z}/p\mathbb{Z}$ is the trivial $G$-module, since $(\mid Aut(\mathbb{Z}/p\mathbb{Z})\mid,p)=(\mid\mathbb{Z}/(p-1)\mathbb{Z}\mid,p)=1,$ therefore
$M_G=M\neq0$.

Let $n=k$, then submodule of elements fixed by $G$ in $M$ must be nontrivial and contains $M_0=\mathbb{Z}/p\mathbb{Z}$ with the trivial action of $G$. Denote $M_1=M/M_0$ and let $\psi:M\rightarrow M_1$ be the corresponding factorization homomorphism.
Since $\psi((g-1)M)=(g-1)M_1$,  $\psi$ induces epimorphism $M_G \twoheadrightarrow (M_1)_G$.
But $(M_1)_G\neq 0$ by induction hypothesis, therefore $M_G\neq0$.
\end{proof}

\begin{proposition} \label{p3}
A pro-$p$-presentation \eqref{eq1} of a pro-$p$-group $G$ with one relator $r$ is quasirational.
\end{proposition}
\begin{proof} First note, that $R=(r)_F$ has a basis $I$ converging to 1, which consist of elements $^f$, where $r=f^{-1}rf, f\in F$ and contain $r$. Then elements of the form $^g\overline{r}, g=\overline{f}=\pi(f)\in G$, where we denote $\overline{r}$ the image of $r$ in $\overline{R}$ turn out to be a basis convergent to 0 in $\overline{R}$.

Now we prove that $R/[R,RF]\cong\mathbb{Z}_p$ is a topological normal closure of the image of $r$. Indeed, since the map of $I$ into one point is continuous and $\overline{R}$ is free as abelian pro-$p$-group, then there exists $\tau:R/[R,R]\twoheadrightarrow \mathbb{Z}_p$ - an epimorphism of pro-$p$-groups which sends the basis $\langle^g\overline{r}, g\in G\rangle$ of $\overline{R}$ into the element $\tau(\overline{r})=\psi(1)$, where $\psi:\mathbb{Z}_pG\rightarrow R/[R,RF]$ be the map of taking coinvariants.
$$\xymatrix{
  \mathbb{Z}_pG \ar[rr]^{\phi} \ar[dr]_{\psi}
                &  &    \overline{R} \ar[dl]^{\tau}    \\
                & \mathbb{Z}_p                 }$$

 The composition of $\tau$ and $\phi$ equals $\psi$ as a homomorphism of free Abelian pro-$p$-groups. Since $\phi$ is epi, then $\tau$ must be $G$-module homomorphism. Indeed, finite sums $a=\sum_{i\in I} a_i ^{g_i}\overline{r}$ generate the dense subgroup in $\overline{R},$ and therefore the identities $\tau(g\cdot a)= g\cdot\tau(a) $ imply that $\tau$ is a module homomorphism (for profinite modules one should know the action on finite factors):

$$\sum_{i\in I} a_i=\tau(a)=g\cdot\tau(a)=\tau(g\cdot a)=\tau(g\cdot \sum_{i\in I} a_i\cdot  ^{g_i}\overline{r})=\tau(\sum_{i\in I} a_i\cdot ^{g \cdot g_i}\overline{r})=$$
$$=\mbox{commutativity of the diagram}=\psi(\sum_{i\in I} a_i(g\cdot g_i)=\sum_{i\in I} a_i$$

Since $im(\tau)$ is generated by the image of $\overline{r}$ and the action is trivial, then $R/[R,F]\cong\langle \tau(\overline{r})\rangle\cong\mathbb{Z}_p$. This follows from Lemma \ref{l100} and $dim_{\mathbb{F}_p}(R/R^p[R,F])=1$ for one relator pro-$p$-groups.
Actually, suppose $\overline{R}_n=R/[R,R\mathcal{M}_n]$ has torsion, then $\overline{R}_n\cong M_{tors} \oplus
M_{\mathbb{Z}_p},$ where $M_{tors}$ be a torsion subgroup, $M_{\mathbb{Z}_p}$ be the free $\mathbb{Z}_p$-module. $M_{tors}$ and $M_{\mathbb{Z}_p}$ are $\mathbb{Z}_p[G_n]$-submodules, where
$G_n=F/\mathcal{M}_nR$. Consider $mod(p)$ factor $\overline{R}_n$, then $\overline{R}_n/{p\overline{R}_n}$ (as $\mathbb{F}_p[G/\mathcal{M}_nR]$-module) has the decomposition $\overline{R}_n/{p\overline{R}_n}=M_{tors}/pM_{tors}\oplus M_{\mathbb{Z}_p}/pM_{\mathbb{Z}_p}$ and $M_{tors}\neq 0.$
Therefore $M_{tors}/pM_{tors}\neq 0$ and has exponent equals to $p$ as Abelian group. We have the finite $p$-group $G_n$ which acts on exponent $p$ Abelian group and hence by Lemma \ref{l100} $(M_{tors}/pM_{tors})_G\neq 0$.
$(M_{\mathbb{Z}_p})_G=\mathbb{Z}_p$ and by \cite[4.3]{Se4} $dim_ {\mathbb{F}_p}((M_{\mathbb{Z}_p}/pM_{\mathbb{Z}_p})_G\oplus (M_{tors}/pM_{tors})_G)=dim_{\mathbb{F}_p}(R/R^p[R,F])=1,$ since we have
$dim_{\mathbb{F}_p}((M_{\mathbb{Z}_p}/pM_{\mathbb{Z}_p})_G)=1$ then $dim_{\mathbb{F}_p}(R/R^p[R,F]\geq2$ and therefore a contradiction.
\end{proof}

\begin{definition}\label{d4}\cite[A.2.]{HM2003}
Let us fix a group $G$ (with pro-$p$- topology). Define the (continuous) prounipotent completion of
$G$ as the following universal diagram, in which $\rho$ is a (continuous) Zariski dense homomorphism
from $G$ to the group of $\mathbb{Q}_p$-points of a prounipotent affine group $G_w^{\wedge}$:
$$\xymatrix@R=0.5cm{
                &         G^{\wedge}_w(\mathbb{Q}_p)  \ar[dd]^{\tau}     \\
 G \ar[ur]^{\rho} \ar[dr]_{\chi}                 \\
                &         H(\mathbb{Q}_p)              }$$
We require that for each continuous and Zariski dense homomorphism $\chi$ there exist
a unique homomorphism $\tau$ of prounipotent groups, making the diagram commutative.
\end{definition}

If we consider a finitely generated free group $F(X)$, then, as it is easy to see, its prounipotent
completion possesses the universal properties inherent to a free object, and by analogy with the discrete
or pro-$p$ cases, we shall call such prounipotent group free and denote it by $F_u(X)$.
Interesting relations between completions in the positive and zero characteristics are obtained in \cite{Pri2012}.

In simplicial group theory, by analogy with gluing two-dimensional cells, it is convenient to identify presentation
\eqref{eq1} with the second step of construction of free simplicial (pro-$p$)resolution $F_{\bullet}$ of a (pro-$p$)group $G$
by the ``pas-$\grave{a}$-pas'' method going back to Andre \cite{And}:

\begin{equation}
\xymatrix{
& \ar@<1ex>[r]\ar@<-1ex>[r] \ar[r] & {F(X\cup Y)} \ar@<0ex>[r]^{d_0}\ar@<-2ex>[r]^{d_1} & F(X)
\ar@<-2ex>[l]_{s_0} \ar[r] & G,} \label{2}
\end{equation}
here $d_0, d_1, s_0$ for $x \in X, y \in Y, r_y \in R$ are defined by the identities
$d_0(x)=x,  d_0(y)=1, d_1(x)=x,  d_1(y)=r_y, s_0(x)=x.$

Recall \cite{Mikh2016} that \eqref{2} is a free finite type simplicial (pro-$p$) group,
degenerate in dimensions greater than two. If the pro-$p$- presentation \eqref{eq1} is minimal,
then $$|Y|=dim_{\mathbb{F}_p}H^2(G,\mathbb{F}_p),|X|=dim_{\mathbb{F}_p}H^1(G,\mathbb{F}_p).$$
Let us assign to a finite type simplicial presentation \eqref{2} a presentation of prounipotent groups
as follows (this construction will be called below by the schematization of a presentation).
First, consider the corresponding diagram of group rings
$\xymatrix{kF(X\cup Y)\ar[r]^{d_0} \ar@<-2ex>[r]^{d_1} & kF(X) \ar@<-2ex>[l]_{s_0}
}$.
Then we obtain from \eqref{2}, taking into account finite generation of groups,
using the $I$-adic completion, the following diagram of complete linearly compact Hopf algebras:
$$\xymatrix{\widehat{k}F(X\cup Y)\ar[r]^{d_0} \ar@<-2ex>[r]^{d_1} & \widehat{k}F(X) \ar@<-2ex>[l]_{s_0}
},$$
where $\widehat{k}F(X)= \varprojlim kF(X)/I^n,$ $I$ is the augmentation ideal in $kF(X)$.
Applying Pontryagin duality and antiequivalence of the categories of commutative Hopf algebras and
affine group schemes,
we obtain the diagram of free prounipotent groups
$$\xymatrix{F_u(X\cup Y)\ar[r]^{d_0} \ar@<-2ex>[r]^{d_1} & F_u(X) \ar@<-2ex>[l]_{s_0}}.$$

\begin{definition}
Let us say that we are given a finite type presentation of a prounipotent group $G_u$
if there exist finite sets $X$ and $Y$ such that $G_u$ is included into the following
diagram of free prounipotent groups:

\begin{equation}
\xymatrix{
{F_u(X\cup Y)} \ar@<0ex>[r]^{d_0}\ar@<-2ex>[r]^{d_1} & F_u(X) \ar@<-2ex>[l]_{s_0} \ar[r] & G_u},\label{5}
\end{equation}
in which the identities similar to \eqref{2} and $G_u\cong F_u(X)/d_1(kerd_0)$ hold.
\end{definition}

\emph{Denote $R_u=d_1(kerd_0)$, this is a normal subgroup in $F_u(X)$, and hence we obtain the analog of the
notion of presentation \eqref{eq1} for a prounipotent group $G_u,$ to which we shall refer also, for uniformity,
as to a presentation of type \eqref{eq1}.}

By analogy with the discrete and pro-$p$ cases, the set of rational points of the \emph{relation module}
$\overline{R_u}(\mathbb{Q}_p)=R_u/[R_u,R_u](\mathbb{Q}_p)\cong R_u(\mathbb{Q}_p)/[R_u(\mathbb{Q}_p),R_u(\mathbb{Q}_p)]$
\emph{of the prounipotent presentation} \eqref{5} is endowed with a structure of $\mathcal{O}(G_u)^*$- module
\cite[Proposition 3]{Mikh2016}.

$QR$-presentations can be studied by passing to the rationalized completion
$\overline{R}\widehat{\otimes} \mathbb{Q}_p=\varprojlim_n R/[R,R\mathcal{M}_n] \otimes \mathbb{Q}_p.$
It turns out \cite[Lemma 2]{Mikh2016} that the topological $\mathbb{Q}_p$-vector space
$\overline{R}\widehat{\otimes} \mathbb{Q}_p$ is identified with
$\overline{R^{\wedge}_w}(\mathbb{Q}_p)$ (where $\overline{R^{\wedge}_w}$ is the quotient group of $R^{\wedge}_w$
by the commutant),
and one can define on it the structure of topological $G$- module \cite[Definition 10]{Mikh2016}, where $G$ is the pro-$p$-group given by the
pro-$p$-presentation \eqref{eq1}. Moreover, these modules can be included into the exact sequence
related with the prounipotent module of relations \cite[Theorem 1]{Mikh2016}.

\subsection{Relation modules of prounipotent groups} \label{s1.4}

The celebrated Lyndon Identity theorem states that the relation modules of a discrete group
with one defining relation is induced from a cyclic subgroup. That is, let \eqref{eq1} be a presentation of a group
with one relation, then $$\overline{R}=R/[R,R]\cong \mathbb{Z}\otimes_{\langle u\rangle} \mathbb{Z}G,$$
where $R=(u^m)_F$ and $u$ is not a nontrivial power, $\langle u\rangle$ - the cyclic subgroup generated by $u$.

Presentations of pro-$p$-groups with one relation such that their $mod(p)$ module of relations
$\overline{R}/p\overline{R}=R/R^p[R,R]$ is induced,
$R/R^p[R,R]\cong \mathbb{F}_p\otimes_{\langle u\rangle} \mathbb{F}_pG,$
are (of course) permutational in the sense of the paper \cite{Mel1} (see also \cite[Definition 4]{Mikh2017}), i.~e.
$\overline{R}/p\overline{R}=R/R^p[R,R]\cong \mathbb{F}_p(T,t_0),$ where $(T,t_0)$ is a profinite
$G$-space with a marked point. O.~V.~Melnikov shows \cite[Theorem 3.2]{Mel1} that relation modules
of aspherical pro-$p$-groups with one defining relation are induced from cyclic subgroups,
as in the Lyndon Identity theorem.

In \cite[10.2]{Se} J.-P.~Serre asks: ``Let $r\in F^p[F,F]$, and let $G_r=F/(r)_F.$ Can one extend to $G_r$
the results proved by Lyndon in the discrete case?''

We see that this question (in the modern settings of \cite{Mel1}) is equivalent to the following one: ``Is it true that relation modules of
pro-$p$-groups with one defining relation are permutational?''

If the answer to this question were true, then each pro-$p$-group with one relator,
such that in its presentation \eqref{eq1} the normal subgroup $R=(r)_F$ is not generated by a $p$-the power,
would have, by \cite[Theorem 3.2]{Mel1}, cohomological dimension 2.

Shift of dimension enables one to calculate cohomology of a pro-$p$- group $H$
as invariants of certain modules. From the viewpoint of the category of representations,
in order to look similar to a group with elements of finite order,
it suffices for the elements of the group $H$ to act as elements of finite order,
although these elements can be actually not of finite order.
Using multiplication of the defining relator $r=y^p$ in the free pro-$p$-group
by elements $\zeta^p$ of special kind, in the next Subsection we shall obtain the defining relations
$y^p\cdot \zeta^p$, which act on finite dimensional modules of arbitrarily high given dimension
exactly as the initial relation $r$,
but are not $p$-th powers themselves. This observation is in favor of the assumption that
pro-$p$-groups with one relator, which are not generated by a $p$-the power,
in contrast to the discrete case, can have cohomological dimension greater than 2.

Nevertheless, we can construct, as in \ref{s1.3}, schematization of a pro-$p$- presentation \eqref{5}
(for details see \cite[3.3]{Mikh2016}). For such prounipotent presentation we shall prove the following
prounipotent analog of Lyndon's result, which can be interpreted as an answer to Serre's question \cite[10.2]{Se}.
We need the following definition \cite{Mikh2016}.

\begin{definition}
Let $A$ be a complete linearly compact Hopf algebra over a field $k$
(the field is considered with discrete topology).
By a left (or right) complete topological $A$-module we shall call a linearly compact topological
$k$- vector space $M$ with a structure of $A$-module
such that the corresponding $k$-linear action $ A\widehat{\otimes} M\rightarrow M$
is continuous. We assume that the topology on $M$ is given by a fundamental system of neighborhoods of zero
$M=M^0\supseteq M^1 \supseteq M^2\supseteq    \ldots,$ where $M^j$ are topological $A$-submodules in
$M$ of finite codimension (finiteness of type) and $M\cong \varprojlim M/M^j.$
By a homomorphism of topological $A$-modules one calls a continuous $A$- module homomorphism.
If the filtration $M^j$ admits a compression such that for each $j$
the compression quotients $M^{j_i}_i/M^{(j+1)_{i}}_{i+1}$ are trivial $A$- modules (i.~e. the action of $A$ is trivial),
then we shall call such topological $A$-module prounipotent.
 \end{definition}

In what follows, prounipotent topological modules will take its origin from a prounipotent group acting by conjugations on its normal subgroup. And we use the coinvariant filtration, that is Zariski closure $M^j=\overline{(g-1)M},$ where $g\in C_k(G)$ and we denoted by $C_k(G)$ the lower central series filtration of a prounipotent group $G$ and $A=\mathcal{O}(G)^*$ (see \cite[p.11]{Mikh2016}).

\begin{theorem}[Identity theorem for pro-$p$-groups] \label{t02}
Let $G$ be a pro-$p$-group with one defining relation given by a finite  presentation \eqref{eq1} in the category of pro-$p$-groups.
Then one has the isomorphism of prounipotent topological $\mathcal{O}(G_u)^*$-modules
$\overline{R_u}(\mathbb{Q}_p)\cong\mathcal{O}(G_u)^*,$
 where the prounipotent groups $G_u$ and $\overline{R_u}$ are obtained from the schematization of the initial pro-$p$-presentation.
\end{theorem}

The analogy with the celebrated Lyndon Identity theorem comes from the fact that $G_u(\mathbb{Q}_p)$ generates $\mathcal{O}(G_u)^*$ as the topological vector space and therefore should be considered as a permutational basis of $\overline{R_u}(\mathbb{Q}_p)$. We can always identify $G_u$ with the continuous $\mathbb{Q}_p$-prounipotent completion of $G$ \cite[(4)]{Mikh2016}.
This result (Theorem \ref{t02}) has been announced in \cite[Corollary 12]{Mikh2015} and used in
\cite[Proposition 3, Corollary 3]{Mikh2016} for proof of the criterion of cohomological dimension
equal to 2,
providing a relation with the known group theory results (see \cite{Lab,Rom,Ts,GIK} and other papers cited there).
Let us recall the results of \cite[Proposition 3, Corollary 3]{Mikh2016}, since they shed light on the following
Serre's question from \cite{Se}.

Let $G_r=F/(r)_F,$ where $(r)_F$ is the normal closure of $r\in F^p[F,F]$ in a free pro-$p$-group
$F$ of finite rank,
then J.-P.~Serre asks:
``Can it be true that $cd(G_r)=2,$ if only $G_r$ is torsion free (and $r\neq1$)?"

By the well known Malcev theorem, a finitely generated discrete nilpotent group without torsion
is embedded \cite[4]{Vez} into its rational prounipotent completion. For $\mathbb{F}_p$-prounipotent groups,
whose instances are pro-$p$-groups, it would be too optimistic to hope for a similar statement, however we have

\begin{proposition}\label{p01} \cite[Proposition 3]{Mikh2016}
Let $G$ be a finitely generated pro-$p$-group given by presentation \eqref{eq1}
with one defining relation $r\neq1$, and assume that the natural homomorphism from $G$ to the group of $\mathbb{Q}_p$-points
$G^{\wedge}_w(\mathbb{Q}_p)$ of its prounipotent completion is an embedding, then
$cd(G)=2.$
\end{proposition}

Let us say that a pro-$p$-group $G$ is $p$-regular if for any finite quotient
$G/\mathcal{M}_n(G)$ there exists a nilpotent quotient $G/V$ without torsion such that $V\subset\mathcal{M}_n(G)$.
Let us also say that a pro-$p$-group $G$ has a $p$-regular filtration $\mathcal{V}=(V_n,n \in \mathbb{N})$
if for any finite quotient $G/\mathcal{M}_k(G)$ there exists $m(k) \in \mathbb{N}$ such that
$V_{m(k)}\subset\mathcal{M}_k(G)$ and the quotients of this filtration are torsion free.

\begin{proposition} \label{s4} \cite[Corollary 3]{Mikh2016}
Let $G$ be a $p$-regular pro-$p$-group with one relation, then $cd(G)= 2$. In particular, if $G$ has a
$p$-regular filtration $\mathcal{V}$, then $cd(G)= 2$.
\end{proposition}

\subsection{Tannaka duality and conjurings.}\label{s1.5}
 Let $\omega:Rep_k(G)\rightarrow Vec_k$ be the erasing functor from the category of representations
 of an affine group scheme $G$ to the category of vector spaces. Then by definition elements of
 $\underline{Aut}^{\otimes}(\omega)(A)$, for any $k$-algebra $A$, are families
 $(\lambda_X), X\in ob(Rep_k(G)),$ where $(\lambda_X)$ are
 $A$-linear automorphisms of $A$-modules $X\otimes A$ such that

 \begin{description}
\item[(i)]{$\lambda_{X\otimes Y}=\lambda_X\otimes \lambda_Y$.}
\item[(ii)]{$\lambda_1= id_A$.}
\item[(iii)]{$\lambda_Y\circ(\alpha \otimes 1)=(\alpha \otimes 1)\circ\lambda_X$
for any $G$-equivariant $k$-linear maps $\alpha:X\rightarrow Y$.}
\end{description}

The Tannaka duality establishes \cite[Proposition 2.8]{DM} the isomorphism of functors
$G\rightarrow \underline{Aut}^{\otimes}(\omega)$ on the category of $k$-algebras.
Thus, an affine group scheme can be reconstructed from its category of representations,
and hence the properties of an affine group scheme are determined by its representations.

Pro-$p$-groups are $\mathbb{F}_p$-points of prounipotent $\mathbb{F}_p$-group schemes \eqref{s1.2}.
Our conjecture states that among pro-$p$-groups $G$ with one relator,
which are not embedded into their prounipotent completions, there are those which are torsion free but $cd(G)>2$.
It turns out that one can hide torsion of relations without changing the behavior of any unipotent
representation $F\rightarrow GL_n(\mathbb{Q}_p)$ and $F\rightarrow GL_n(\mathbb{F}_p)$ for arbitrarily
large fixed $n\in \mathbb{N}.$ To be more precise, one has the following

\begin{theorem} \label{p02}
Let $r=w^{p^l}$ be an element of the free pro-$p$-group $F=F_p(d)$ of the rank $d\geq 2$ which is a $p^l$-th power and fix $n\in \mathbb{N}.$ Then there exist elements $z_n \in F$ (``conjurings'')
with the following properties:

\begin{description}
\item[a)]{for any unipotent representations
$\phi:F\rightarrow GL_n(\mathbb{Q}_p)$ and $\psi:F\rightarrow GL_n(\mathbb{F}_p)$, one has
$\phi(z_n)=1, \psi(z_n)=1.$}
\item[b)]{ $r\cdot z_n^p$ is not a $p$-th power.}

\end{description}
\end{theorem}

In Subsection \ref{s2.4} we give, following the ideas of ~\cite{Mag}, construction of the Magid
identities for free prounipotent groups over arbitrary zero characteristics fields and prove the Theorem \ref{p02}.
In the presented article only the construction of conjurings is described, we plan to return to applications in subsequent works.

\section{Cohomology and presentations}\label{s3}

Cohomology theory of prounipotent groups over the algebraically closed field $k$ has been developed by A. Lubotzky and A. Magid \cite{LM1982}, when $k$ has zero characteristics. This theory closely parallels that of the cohomology of pro-$p$-groups \cite[I,4]{Se4}: the free prounipotent groups turn out to be those of cohomological dimension ones, the dimension of the first and second cohomology groups give numbers of generators and defining relations. In addition, authors have proved that one-relator prounipotent groups turn out to have cohomological dimension two, similar to discrete one-relator groups those have not a power elements as defining relations. Schematization leads to prounipotent groups over not algebraically closed fields ($k=\mathbb{Q}_p$ in our case) and hence we want to develop a similar theory. As was rigorously pointed out the author by Richard Hain, nothing new in this case can happen.
Indeed, let $k$ be a non algebraically closed field of zero characteristics with algebraic closure $\overline{k}$ and consider unipotent group $G_k$ over $k$. We already know that $G(k)\cong k^{\alpha}$ for some $\alpha\in \mathbb{N}$ and \cite[4.1]{Wat} imply that the closure $G_{\overline{k}}$ of $G_k$ in $\overline{k}^{\alpha}$ is just $\overline{k}^{\alpha}$. It turns out that $\mathcal{O}(G_{\overline{k}})\cong\mathcal{O}(G_k)\otimes_k \overline{k}$ and \cite[Proposition 4.18]{Jan} imply that $H^*(G_{\overline{k}},\overline{k})\cong H^*(G_k,k)\otimes_k \overline{k}.$
Below we introduce necessary notions for defining Hochschild cohomology groups of affine group schemes in the modern setting
\cite{Jan}. In the case of algebraically closed field they coincide \cite[p.~28]{Jan} with the ones defined in ~\cite{LM1982}.

\subsection{Modules and cohomology}\label{s3.1}

Let $G$ be an affine group scheme, and let $Rep(G)$ be the corresponding category of $G$-modules \cite[I,2.7,2.8]{Jan},
identified using the Yoneda lemma with the category of $\mathcal{O}(G)$-comodules \cite[I,2.8]{Jan}.
Each $G$-module $M$ is representable as the inductive limit of its finite dimensional submodules \cite[I,2.13(1)]{Jan}.
Recall, following \cite[I,2.2]{Jan}, that the notion of module is functorial. To this end, let us define for each
finite dimensional $k$-vector space $M$ certain $k$-group functor $M_a$ by the rule
$M_a(A)=(M\otimes A,+)$ for all $k$-algebras $A$. In our case, namely when $k$ is a field,
and dimension of $M$ is finite, $M_a$ is representable by the symmetric algebra of the dual $k$-module $M^*$,
which we denote by $S(M^*)$, then $\mathcal{O}(M_a)=S(M^*)$ \cite[3.6]{Milne}.

Denote by $Mor(G,M_a)$ the $k$-space of natural transformations of functors. On $Mor(G,M_a)$
one has left regular action of $G$, given by the formula
$(x\cdot f_R)(g)=f_R(gx),$ where $g\in G(R),f_R\in Mor(G(R),M_a(R)),x\in G(R).$ Let us introduce also right regular action, given by the formula
$f_R\cdot x=f_R(xg).$ If $M=k_a$ is the additive group of the field $k$ \cite[3.1]{Milne}, then
$Map(G,k_a)=\mathcal{O}(G)$ is the coordinate ring of the affine group scheme $G$ \cite[2.15]{Milne}.
For constructing injective envelopes we shall need the notion of induced module.
Thus, let $H$ be a subgroup of the affine group scheme $G$. For each $H$-module $M$ define
the induced module $M\uparrow^G_H$ as follows:
$$M\uparrow^G_H=\{ f\in Mor(G,M_a)\mid f(gh)=h^{-1}f(g),\forall g\in G(A),h\in H(A), A\in Alg_k\}, $$
where $G$ acts regularly on the left. In \cite[I,3.3]{Jan}, using the identification
$M\uparrow^G_H\cong (M\otimes \mathcal{O}(G))^H,$
$M\otimes \mathcal{O}(G)$ is endowed with a structure of $(G\times H)-$module, where the action of $H$ on $M$ is given,
and the action on $\mathcal{O}(G)$ is through the right regular representation; $G$ acts on $M_a$ trivially
and on $\mathcal{O}(G)$ from the left through the left regular representation;
one takes the tensor product of representations,
and it is shown that $M\uparrow^G_H$ is indeed a $G-$module.

For arbitrary $k$-module $M$, let $\varepsilon_M:M\otimes \mathcal{O}(G)\rightarrow M$
be a linear map $\varepsilon_M=id_M\otimes \varepsilon_G. $

\begin{proposition} [Frobenius reciprocity] ~\cite[I,3.4]{Jan} \label{p4}
Let $H$ be a closed subgroup of an affine group scheme $G$ and $M$ be an $H$-module.

a) $\varepsilon_M:M\uparrow^G_H\rightarrow M$ is a homomorphism of $H$-modules

b) For each $G-$module $N$ the map $\varphi\mapsto \varepsilon_M\circ \varphi $ defines an isomorphism
$$Hom_G(N,M\uparrow^G_H)\cong Hom_H(N\downarrow^G_H,M).$$
\end{proposition}

\begin{proposition}[Tensor identity] ~\cite[I,~3.6]{Jan} Let $H$ be a closed subgroup of an affine group scheme
$G$ and $M$ be an $H$-module. If $N$ is a $G$-module, then there is a canonical isomorphism of $G$-modules
$$(M\otimes N\downarrow^G_H)\uparrow^G_H\cong M\uparrow^G_H\otimes N.$$
\end{proposition}

Following ~\cite[I,~3.7]{Jan}, let us emphasize some useful corollaries from the propositions given above.
Assume that $H=1$, then
$M\uparrow^G_1=M\otimes \mathcal{O}(G), $
for any $k$-module $M$, and in particular
$k\uparrow^G_1=\mathcal{O}(G).$

Combining the latter identity with the Frobenius reciprocity b) we obtain, for each $G$-module $M$,
$$Hom_G(M,\mathcal{O}(G))\cong M^*.$$

If we put $M=k_a$ in the tensor identity, then for each $G$-module $N$ we obtain the remarkable isomorphism
$$N\otimes \mathcal{O}(G)\cong N\uparrow^G_1=N_{tr}\otimes \mathcal{O}(G),$$
given by the formula
$x\otimes f\mapsto (1\otimes f)\cdot(id_N\otimes\sigma_G)\circ\Delta_N(x),$ where
we have denoted by $N_{tr}$ the $k$-module $N$ with the trivial action of $G,$ and $\sigma_G$ is the antipode in
$\mathcal{O}(G)$.

We shall need $M^G$, the submodule of fixed points of a $G$-module $M$:
$$M^G=\{m\in M\mid g(m\otimes1)=m\otimes1, \forall g\in G(A),A\in Alg_k\}.$$
If in the definition we take $g=id_{\mathcal{O}(G)}\in G(\mathcal{O}(G))$, then we obtain
$$M^G=\{m\in M\mid \Delta_M(m)=m\otimes1\}.$$

Consider the regular representation of an affine group scheme $G:$
$$\Delta: \mathcal{O}(G)\rightarrow \mathcal{O}(G)\otimes \mathcal{O}(G).$$
If $H$ is a closed subgroup of $G$, then $\mathcal{O}(H)=\mathcal{O}(G)/I_H,$
where $I_H$ is the Hopf ideal defining the subgroup $H$, whence we obtain the $k$-linear map
$$\mu: \mathcal{O}(G)\rightarrow \mathcal{O}(G)\otimes \mathcal{O}(H)$$ which defines the structure
of a right $H$-module on the algebra $\mathcal{O}(G)$.

Since the category of $G$-modules is Abelian ~\cite[I,~2.9]{Jan} and since, due to \cite[I,~3.9]{Jan},
it contains enough injective objects, one can define cohomology groups $H^n(G, M)$ of an affine group scheme
$G$ with coefficients in a $G$-module $M$ as the $n$-th derived functors of the fixed points functor $( )^G$
computed for $M$. Note the possibility of computing $H^n(G, M)$ by means of the Hochschild complex
$C^*(G, M)$ ~\cite[I,~4.14]{Jan}.

Cohomology well behaves with respect to limits. Let us represent the affine group scheme $G$ as
$G\cong\varprojlim G_{\alpha}$, where $G_{\alpha}$ are affine algebraic group schemes.
Since $\mathcal{O}(G) \cong \varinjlim \mathcal{O}(G_{\alpha})$ and since each $G$-module $V$
can be represented as a direct limit of finite dimensional $G$-submodules $V_{\beta}$ \cite[Theorem 3.3]{Wat}, then
$(V\otimes \mathcal{O}(G)^{\otimes^n})^G\cong \varinjlim (V_{\beta}\otimes \mathcal{O}(G_{\alpha})^{\otimes^n})^{G_{\alpha}}.$
Homology commutes with direct limits, hence
$$H^n(G, V )\cong \varinjlim H^n(G_{\alpha}, V_{\beta}).$$
Note that pro-$p$-groups can be considered as $\mathbb{F}_p$-points of prounipotent groups over the field
$\mathbb{F}_p$ (for constructing the $\mathbb{F}_p$-Hopf algebra it suffices to consider the decomposition of the
$\mathbb{F}_p$-group ring of a pro-$p$-group into the projective limit of the group rings of finite $p$-groups,
consider the dual Hopf algebras, and take their inductive limit), then cohomology of pro-$p$-groups
in the sense of ~\cite[6.6]{ZR} coincides with the cohomology of the corresponding affine group schemes (this identification may be deduced from \cite[4.4]{Jan} and Pontryagin duality \cite[Proposition 6.3.6]{ZR}).

\begin{proposition} \label{p3} ~\cite[~1.10]{LM1982} Let $$1 \rightarrow H \rightarrow K \rightarrow G \rightarrow 1$$
be an exact sequence of affine group schemes, and let $M$ be a $G$-module. Then there exists
a spectral sequence with the initial term $E_2^{p,q}=H^p(G, H^q(H, N)),$
converging to $H^{p+q}(K, M).$
\end{proposition}

In the case the affine group scheme $G$ is prounipotent, there is a more precise description of injective
modules.

 \begin{proposition} \label{p8} ~\cite[~1.11]{LM1982}
Let $G$ be a prounipotent group, and let $V$ be a $G$-module. Then
$V^G\otimes \mathcal{O}(G)$ is an injective $G$-module containing $V$.
Each injective $G$-module containing $V$ contains a copy of $V^G\otimes \mathcal{O}(G)$.
\end{proposition}

Due to the previous statement we can define the injective envelope of a $G$-module $V$ by the formula
$\mathcal{E}_0(V) = V^G\otimes \mathcal{O}(G).$ Put $\mathcal{E}_{-1}(V)=V$, and let
$d_{-1}:\mathcal{E}_{-1}(V)\rightarrow \mathcal{E}_0(V)$ be the corresponding inclusion.

 \begin{proposition} \label{p9}
 ~\cite[~1.12]{LM1982} Let $G$ be a prounipotent group and $V$ be a $G$-module. Define the
 minimal resolution $\mathcal{E}_i(V)$ and $ d_i : \mathcal{E}_i(V)\rightarrow \mathcal{E}_{i+1}(V)$ inductively,
 $$\mathcal{E}_{i+1}(V) = \mathcal{E}_0(\frac{\mathcal{E}_i(V)}{d_{i-1}(V))})\quad d_i=\mathcal{E}_i(V)\rightarrow
 \frac{\mathcal{E}_i(V)}{d_{i-1}(\mathcal{E}_{i-1}(V))}\rightarrow \mathcal{E}_{i+1}.$$
 Then $\{\mathcal{E}_i( V), d_i \}$ is an injective resolution of $V$ and $H^i(G, V) = \mathcal{E}_i( V)^G.$
\end{proposition}

We shall say that an affine group scheme $G$ has cohomological dimension $n$
and write $cd(G) = n$, if for any $G$-module $V$ and for each $i > n, H^i(G, V) = 0$ and $H^n(G, V)\neq 0$.
If $G$ is prounipotent, then $cd(G) \leq n$ if and only if $H^{n+1}(G, k_a) = 0$, since
$k_a$ is the only simple $G$-module.

 \begin{proposition} ~\cite[~1.14]{LM1982} Let $G$ be a prounipotent group and $H$ be a subgroup.
 For any $H$-module $V$ there is an isomorphism
 $H^n(G, V\uparrow^G_H) \cong H^n(H, V)$ for all $n\in \mathbb{N}$. In particular, $cd(H) \leq cd(G).$
\end{proposition}

Since the following statement is important for the succeeding exposition and for demonstration how
Pontryagin duality helps one to transfer arguments from \cite{LM1982} into our situation,
let us give this statement with a full proof. Denote by $Hom(G,k_a)$ the set of affine group scheme homomorphisms from $G$ to $k_a$.
Let $V$ be a linearly compact topological vector space, then by $V^{\vee}=Hom_{cts}(V,k)$ we denote the set of continuous (in the linearly compact topology) maps of vector spaces from $V$ to $k$. We also note that for any Abelian unipotent group $G$ over a zero characteristics field $k$ one has $G(k)\cong k^n$ for some $n\in \mathbb{N}.$ Hence $G(k)$ could be considered as a $n$-dimensional vector space over $k$. Consequently $G(k)$ has a structure of pro-finite dimensional (linearly compact) topological vector space over $k$.

\begin{proposition} \label{p11}
 ~\cite[~1.16]{LM1982} Let $G$ be a prounipotent group, then:
 
1) $H^1(G, k_a) \cong Hom(G,k_a)$;

2) if $G$ is Abelian, then $G(k)^{\vee}=Hom_{cts}(G(k),k)\cong Hom(G,k_a)$.

\end{proposition}
\begin{proof}
1) Proposition \ref{p8} shows that the beginning of the minimal injective resolution of the trivial $G$-module $k_a$
(or equivalently of the trivial $\mathcal{O}(G)$-comodule) has the form $k_a\rightarrow \mathcal{O}(G),$
and hence, due to triviality of differentials on the fixed points of the minimal resolution (Proposition \ref{p9}),
elements of $H^1(G, k_a)$ correspond to $G$- invariant modulo $k_a$ elements of $\mathcal{O}(G)$.
The Pontryagin duality \cite[3]{Mikh2016} enables one to consider an element $a\in \mathcal{O}(G)$
as a $k$-linear function
$a: \mathcal{O}(G)^*\rightarrow k,$
continuous in the linearly compact topology on $\mathcal{O}(G)^*.$ Recall that $k$ is considered with the
discrete topology, and the regular action of $G(k)$ on $\mathcal{O}(G)$ is the action on functions by the formula
$(g\cdot a)(x)=a(x\cdot g), x\in\mathcal{O}(G)^*.$ Then $G(k)$- invariance modulo $k_a$
is written in the form $g\cdot a(x)-a(x)=const\in k $ for $\forall x\in \mathcal{O}(G)^*. $
Without loss of generality one can normalize $a$ putting $a(1)=0.$ Now we define the homomorphism
$f:G(k)\rightarrow k_a$ for $g\in G(k)$ by the formula $f(g)=g\cdot a(x)-a(x).$

Conversely, if we are given a homomorphism of prounipotent groups $f:G(k)\rightarrow k_a$, then
it is extended by linearity to a continuous in the linearly compact topology homomorphism of
topological vector spaces
$f:\mathcal{O}(G)^*\rightarrow k_a.$
Such $f$ is invariant modulo elements of the field. Indeed,
$(g\cdot f)(x)=f(x\cdot g)=f(x)+f(g)=f(x) \quad mod(k_a).$ 

2) Since $G$ is Abelian, the functorial correspondence between prounipotent groups and pronilpotent Lie algebras gives rise to the isomorphism $$Hom(G,k_a)\cong Hom_{Lie}(log(G(k)),k).$$ But commutative pro-finite dimensional Lie algebra $log(G(k))$ and $G(k)$ are the same as linearly compact topological vector spaces, hence $Hom_{Lie}(log(G(k),k)\cong Hom_{cts}(G(k),k)= G(k)^{\vee}.$
\end{proof}

\subsection{Presentations of prounipotent groups} \label{s3.2}
Below let $k=\mathbb{Q}_p$ or another field of characteristics zero, and let us identify prounipotent groups
with their groups of $k$-points (see Subsection \ref{s1.2}).

Let $Z$ be a set and $F(Z)$ be the free group on the set $Z$. Denote by $F(Z)^{\wedge}_u$
its $k$-prounipotent completion and $\rho :Z\rightarrow F(Z)^{\wedge}_u$ be the natural embedding. Let us construct a prounipotent group $F_u(Z)$ equipped with an inclusion
$i: Z\rightarrow F_u(Z)$, which will be called the free prounipotent group on the set $Z$.

Let $\mathcal{L}$ be the set of all normal subgroups $H \trianglelefteq F(Z)^{\wedge}_u$ of finite codimension
and such that the set $\{x \in Z \mid \rho(x) \notin H\}$ is finite. If $H_1, H_2 \in \mathcal{L}$, then
$H_1 \cap H_2 \in \mathcal{L}$.
Let $K = \cap \{ H \mid H \in \mathcal{L}\}$, then put $F_u(Z) = F(Z)^{\wedge}_u/K $. The definition shows that
$$F_u(Z)\cong \varprojlim \{ \frac{F(Z)^{\wedge}_u}{H} \mid H \in \mathcal{L}\}.$$
Put $i: Z\rightarrow F_u(Z)$; this is the composition of $\rho$ with the canonical homomorphism
$F(Z)^{\wedge}_u\rightarrow F_u(Z)$. It is not difficult to check \cite[p.83]{LM1982} that $i$ is an injective map.

\begin{definition} Let $Z$ be a set, then the prounipotent group $F_u (Z)$  will be called
the free prounipotent group on the set $Z$. Using $i$ we identify $Z$ with a subset of $F_u(Z)$.
\end{definition}

\begin{proposition} \label{p7} Let $Z$ be a set, and $G$ be a unipotent group. Then there is a bijection between
the homomorphisms $f:F_u(Z)\rightarrow G$ and the sets of elements $x_i\in G:Card(\{i \in Z  \mid x_i \neq e\})<\infty$,
so that the homomorphisms correspond to the sets $\{f(i) \mid i \in Z\}$.
\end{proposition}

Proposition \ref{p7} implies that the free prounipotent group $F_u(Z)$ has the following lifting property: let
$$1 \rightarrow k_a\rightarrow E \xrightarrow{g} U \rightarrow 1$$ be an exact sequence of unipotent groups and
$f : F_u(Z) \rightarrow U$ be a homomorphism, then there exists a homomorphism $h : F_u(Z) \rightarrow E:gh = f.$

\begin{proposition} \cite[Th. 2.4]{LM1982} Let $G$ be a prounipotent group, then the following conditions are
equivalent:

(a) If $$1 \rightarrow K \rightarrow E \xrightarrow{g} F \rightarrow 1$$ is an exact sequence of
prounipotent groups and
$f : G\rightarrow F$ is a homomorphism, then there exists a homomorphism $h:G\rightarrow E$ such that $gh=f$.

(b) If $$1 \rightarrow k_a \rightarrow E \xrightarrow{g} F \rightarrow 1$$ is an exact sequence of
prounipotent groups and
$f : G\rightarrow F$ is a homomorphism, then there exists a homomorphism $h:G\rightarrow E$ such that $gh=f$.
\end{proposition}

If a group satisfies the conditions of the previous Proposition, then we shall say that such prounipotent group
has the \emph{lifting property}.

\begin{lemma} \cite[Prop.2.8]{LM1982} Let $G$ be a prounipotent group, then there exists a free
prounipotent group $F_u(Z)$ and an epimorphism $f:F_u(Z)\rightarrow G$. The data $Z$ and $f$ can be chosen so that
$Hom_{cts}(G,k_a)$ has the dimension equal to the cardinality of $Z$. Assume that $Y$ is a set and
$g :F_u(Y) \rightarrow G$ is an epimorphism€, then $Card(Y)\geq Card(Z)=Hom_{cts}(G,k_a)$.
If $G$ has the lifting property then $f$ is an isomorphism.
\end{lemma}
The lifting property described above yields the following cohomological description of free prounipotent groups.

\begin{lemma} \label{l1} \cite[Th. 2.9]{LM1982} A prounipotent group $G$ is free if and only if
$cd(G) \leqslant 1.$
\end{lemma}

Now it is not difficult to obtain the statement used in \cite[Corollarry 3]{Mikh2016}.

\begin{proposition} \cite[Corollary 2.10]{LM1982} Let $H$ be a subgroup of a free prounipotent group $G$, then $H$ is free.
\end{proposition}

\begin{definition} A prounipotent group $G$ is called finitely generated if there exists a set of elements
$\{ g_1,..., g_n \}$ in $G$ such that the abstract subgroup in $G$ generated by $g_1, ,..., g_n$
is Zariski dense in $G$. In this case we say that $\{g_1 ,..., g_n\}$ is a set of generators in $G$.
If $G$ is finitely generated, then by the rank of $G$ we mean the minimal cardinality of a set of generators.
\end{definition}

\begin{theorem} \label{t04} A prounipotent group $G$ is finitely generated if and only if
$H^1(G, k_a)$ has finite dimension. If $G$ is finitely generated then the rank of $G$ equals
the dimension of $H^1(G, k_a).$
\end{theorem}

\begin{definition} Let $G$ be a prounipotent group and $N$ be its normal subgroup.
Let us say that $N$ is finitely related (as a normal subgroup) if there exists a set of elements
$\{g_1, ,..., g_n\}$ in $N$ such that the abstract subgroup of $N$ generated by all $G$-conjugations
of $g_i$ is Zariski dense. If $n$ is minimal then $n$ is called the minimal \emph{number of relators} of $N$.
\end{definition}

\begin{definition} \label{d7}
Let us call by a \emph{proper presentation} of a prounipotent group $G$ an exact sequence \eqref{eq1}
of prounipotent groups in which $F$ is free, and the homomorphism $H^1(G, k_a)\rightarrow H^1(F, k_a)$
is an isomorphism.
Let us say that $G$ is given by a finite number of relations if there exists such a sequence in which
$R$ is finitely related as a normal subgroup of $F$.
\end{definition}
\begin{definition} We shall say that a prounipotent group $G$ has $n$ \emph{relations} if in any proper
presentation \eqref{eq1} the normal subgroup $R$ is finitely related
as a normal subgroup of $F$, with the minimal number of relators equal to $n$.
\end{definition}

\begin{theorem} \cite[Th. 3.11]{LM1982} The prounipotent group $G$ has a finite number of relations
if and only if $H^2(G, k_a)$ is finite dimensional. If $G$ has a finite number of relations and if \eqref{eq1} is any proper
presentation of $G$, then $R$ is finitely related as a normal subgroup of $F$, and its minimal number of
relators is the dimension of $H^2(G, k_a).$
\end{theorem}

\begin{proposition} \cite[Cor. 3.13]{LM1982} \label{p10} A prounipotent group $G$ has $n$ relations if and only if
$H^2(G, k_a)$ has dimension $n$.
\end{proposition}

\begin{proposition} \cite[Theorem 3.14]{LM1982} \label{p18} Let $G$ be a prounipotent group, and assume that for some $n > 1$, $H^n(G, k_a)$
has dimension one, then $cd(G) = n.$
\end{proposition}
\begin{proof}
Let $\mathcal{E}_i, i \in \mathbb{N}$ be a minimal injective $G$-module resolution of $k_a$, then Proposition
\ref{p9} yields $H^i(G, k_a) \cong \mathcal{E}_i^G$ and by Proposition \ref{p8} $\mathcal{E}_i\cong H^i(G, k_a)\otimes \mathcal{O}(G)$.
In particular, $\mathcal{E}_n =\mathcal{O}(G).$ Since $\mathcal{E}_n\neq 0$, then
$d_{n-1} : \mathcal{E}_{n-1}\rightarrow \mathcal{E}_n$ is a nonzero map.
Any nonzero $G$-endomorphism of $\mathcal{O}(G)$ is onto, since this is true for prounipotent groups obtained by extending
scalars to the algebraic closure $\overline{k}$ \cite[Theorem ~5.2]{LM1982} and $\mathcal{O}(G_{\overline{k}})\cong\mathcal{O}(G_k)\otimes_k \overline{k}$.
Therefore $d_{n-1}$ is an epimorphism, and hence $\mathcal{E}_{n+1}=0$ and $cd(G) = n.$

\end{proof}

\section{Proof of Theorem \ref{t02}}\label{s2.1}

Consider a proper
($dim_{\mathbb{F}_p}H^1(G,\mathbb{F}_p)=dim_{\mathbb{F}_p}H^1(F,\mathbb{F}_p)$)
presentation \eqref{eq1} of a pro-$p$-group $G$ with one defining relation. A proper presentation of $G$
can give rise to a non-proper (Definition \ref{d7}, $dim_{\mathbb{Q}_p}H^1(G_u,\mathbb{Q}_p)<dim_{\mathbb{Q}_p}H^1(F_u,\mathbb{Q}_p),$)
presentation of the prounipotent group $G_u=F_u(X)/R_u,$ where $R_u=(r)_{F_u(X)}$ is the Zariski
closure of the normal subgroup abstractly generated by the element $r.$
Non-properness of the presentation is equivalent to the statement that the element $r$
is a generator of the free prounipotent group $F_u=F_u(X)(\mathbb{Q}_p)$
(here we have identified the prounipotent group $F_u(X)$ with its group of $\mathbb{Q}_p$-points,
which is correct by the arguments from Subsection \ref{s1.2}).
But Theorem \ref{t04} (see also \cite[p.~85--86]{LM1982})
implies that this is equivalent to non-triviality of the image $\phi(r)$ of the relation $r \in F_p \subset F_u$
under the homomorphism $\phi: F_u \rightarrow F_u/[F_u,F_u]$.
\begin{proof}[of Theorem \ref{t02}]

1) Consider the degenerate case. Without loss of generality we can assume that the prounipotent presentation has
the following form:
$$1\rightarrow R_u=(z)_{F_u(X\cup \{z\})}\rightarrow F_u(X\cup \{z\})\xrightarrow{d_0} F_u(X)\rightarrow 1,$$
where $X$ is the free basis of $F_u$. Consider the 2-reduced simplicial group
$$\xymatrix{
{F_u(X\cup \{z\})} \ar@<0ex>[r]^{d_0}\ar@<-2ex>[r]^{d_1} & F_u(X) \ar@<-2ex>[l]_{s_0} \ar[r] & G_u\cong F_u(X),}$$
here $d_0, d_1, s_0$ are defined on $x \in X, z$ by the identities:
$d_0(x)=x,  d_0(z)=1, d_1(x)=x,  d_1(z)=1, s_0(x)=x.$

It is clear that $R_u\cong kerd_0, F_u(X\cup \{z\})\cong kerd_0\leftthreetimes F_u(X).$
Finally, we can apply the arguments from \cite[Proposition~2]{Mikh2016}, showing that
$(kerd_0,F_u(X),d_1|_{kerd_0}=1)$ is a free prounipotent pre-crossed module.
It remains to note that $kerd_1= kerd_0,$ and hence
$\overline{C}_u=kerd_0/[kerd_0,kerd_0kerd_1]=kerd_0/[kerd_0,kerd_0]\cong \overline{R}_u.$
It remains to use \cite[Corollary 3]{Mikh2016}, which implies the required isomorphism
of topological $\mathcal{O}(G_u)^*$-modules $\overline{C}_u(\mathbb{Q}_p)\cong \mathcal{O}(G_u)^*.$

2) Now consider the case in which the pro-$p$-presentation \eqref{2} yields a proper prounipotent
presentation \eqref{5} of the prounipotent group $G_u.$ Now $r \in F_p \subset F_u:=F_u(X)(\mathbb{Q}_p)$
is not a generator in $F_u$. Let us perform the proof by a series of reductions.

First, note that the proof of isomorphism of left topological $\mathcal{O}(G_u)^*$-modules
$\overline{R_u}(\mathbb{Q}_p)\cong \mathcal{O}(G_u)^*$, by Pontryagin duality \cite[3]{Mikh2016}, is equivalent
to the proof of the isomorphism of $\mathcal{O}(G_u)$-comodules
$Hom_{cts}(\overline{R_u}(\mathbb{Q}_p),k)\cong\mathcal{O}(G_u).$

Proposition \ref{p11} states that there is an isomorphism of left
$\mathcal{O}(G_u)$-comodules $$H^1(R_u,k_a)\cong Hom(\overline{R_u},k_a)\cong Hom_{cts}(log(\overline{R_u}(\mathbb{Q}_p)),k)=\overline{R_u}(\mathbb{Q}_p)^{\vee}.$$
Thus, we need to prove the isomorphism of left $\mathcal{O}(G_u)$-comodules
$H^1(R_u,k_a)\cong \mathcal{O}(G_u).$

Let us study the minimal injective $\mathcal{O}(G_u)$- resolution of the trivial left
$\mathcal{O}(G_u)$-comodule $k_a.$ Proposition \ref{p9} implies that since the cohomological dimension
of $F_u$ equals one (Lemma \ref{l1}), then the minimal $\mathcal{O}(F_u)$- resolution of $k_a$ will have the form
$$0\rightarrow k_a\rightarrow \mathcal{O}(F_u)\rightarrow \mathcal{O}(F_u)^{dim_kH^1(F_u,k_a)}\rightarrow 0.$$

\cite[I, Proposition 4.12, Proposition 3.3]{Jan} implies that we can consider this resolution as
an injective resolution consisting of left $\mathcal{O}(R_u)$- comodules. Applying the functor of
$R_u$-fixed points,
we obtain an exact sequence (since $R_u$-fixed points of $\mathcal{O}(F_u)$ coincide with $\mathcal{O}(G_u)$
\cite[16.3]{Wat})
$$0\rightarrow k_a\rightarrow \mathcal{O}(G_u)\rightarrow \mathcal{O}(G_u)^{dim_kH^1(F_u,k_a)}\rightarrow
H^1(R_u,k_a)\rightarrow 0.$$
Taking into account that the presentation is proper, in small dimensions the Grothendieck
spectral sequence (Proposition \ref{p3}) is written in the form
$$1\rightarrow H^1(G_u,k_a)\rightarrow H^1(F_u,k_a)\rightarrow H^1(R_u,k_a)^F\rightarrow
H^2(G_u,k_a)\rightarrow H^2(F_u,k_a)=1,$$
which yields an isomorphism $H^1(R_u,k_a)^G \cong H^2(G_u,k_a).$  Hence, Proposition \ref{p8} shows that the injective
envelope of $H^1(R_u,k_a)$ coincides with $\mathcal{O}(G_u)^{dim_kH^2(G_u,k_a)}$.
Now, considering the composition of the map
$\mathcal{O}(G_u)^{dim_kH^1(G_u,k_a)}\rightarrow H^1(R_u,k_a)$
with the inclusion of $H^1(R_u,k_a)$ into its injective envelope, we obtain the beginning of the
minimal injective $\mathcal{O}(G_u)$-resolution
$$0\rightarrow k_a\rightarrow \mathcal{O}(G_u)\rightarrow \mathcal{O}(G_u)^{dim_kH^1(G_u,k_a)}\rightarrow
\mathcal{O}(G_u)^{H^2(G_u,k_a)}$$
of the trivial comodule $k_a$. In particular, one has an isomorphism of left $\mathcal{O}(G_u)$- comodules
$H^1(R_u,k_a)\cong im\{\mathcal{O}(G_u)^{dim_kH^1(G_u,k_a)}\rightarrow \mathcal{O}(G_u)^{H^2(G_u,k_a)}\}.$
Therefore conditions a) $H^1(R_u,k_a)$ is $G$-injective and b) $cd(G_u)= 2$ are equivalent.

It remains to prove that a prounipotent group given by a proper presentation with one relation has cohomological
dimension equal to two,  but this is a particular case of Proposition \ref{p8}, since
Proposition \ref{p10} implies the equality $dim_kH^2(G_u,k_a)=1.$

\end{proof}

\section{Identities in free pro-$p$-groups}\label{s2.4}

\begin{definition} By an admissible ring of coefficients one calls a commutative complete local
$k$-algebra $R$ without divisors of zero and with the maximal ideal $m = m_R$
such that $R/m = k$ and $l=dim_k(m/m^2)<\infty$.
\end{definition}

Note that the decomposition $R=\varprojlim R/m^i$ enables one to consider $$GL_n(R) = \varprojlim GL_n(R/m^i)$$
as the projective limit of linear algebraic groups $GL_n(R/m^i)$, and hence as a $k$-affine group scheme
\cite{Wat}. One has the Levi decomposition $GL_n(R/m^i)\cong (I+M_n(m/m^i))\leftthreetimes GL_n(k)$
into the semidirect product of the linear algebraic (reductive) group $GL_n(k)$ and the unipotent group
$I+M_n(m/m^i)$. 

Let $K_i=ker\{GL_n(R) \rightarrow GL_n(R/m^i)\},$ then $K_i\cong I + M_n(m^i)$ and
$$K_1/K_i \cong I + M_n(m/m^i)\cong ker\{ GL_n(R/m^i)\rightarrow GL_n(k)\}.$$
Since $I + M_n(m/m^i)$ are unipotent linear algebraic groups, then
$K_1 \cong \varprojlim (K_1/K_i)$ is a prounipotent group, and in particular one has the Levi decomposition
$GL_n(R) = K_1\leftthreetimes GL_n(k).$
\begin{definition} Let $R$ be an admissible ring of coefficients, then by an $R$- admissible representation
of a prounipotent group $U$ one calls a homomorphism of affine group schemes $\rho: U \rightarrow GL_n(R)$.
\end{definition}


Let $\rho : U \rightarrow GL_n(R) \cong K_1\leftthreetimes GL_n(k))$ be an admissible representation.
Then $\rho^{-1}K_1$ is a closed normal prounipotent subgroup of $U$ of finite codimension, such that
the quotient group $\widetilde{U} = U/\rho^{-1}K_1$ has a faithful representation in $GL_n(k).$
Since $U$ is unipotent,
this means that its image is conjugate to a subgroup consisting of upper triangular matrices,
hence $U$ has the rank no more than $n$. Therefore, an $(n + 1)$-multiple commutator from $U$ belongs to $\rho^{-1}K_1$,
and hence from the viewpoint of identities the study of representations into $K_1$ and into $GL_n(R)$ are equivalent.
Hence below in the existence questions of identities we shall restrict ourselves by representations into $K_1.$

\begin{definition} \label{d11}
By an identity of $d\geq2$ variables with values in a prounipotent group $G$ one calls an element
$u$ of the free prounipotent group $F = F(x_1, ... , x_d)$ with $d$ generators, which lies in the kernel
of any homomorphism $f: F \rightarrow G.$ The set of all identities of $d$ variables with values in
$G$ forms a closed normal subgroup $I(d, G)$ in $F$. The set of identities of $d$ variables with values in
a set of prounipotent groups $\mathcal{G}= G_{\alpha}$ is the normal prounipotent subgroup
$I(d,\mathcal{G}) =\cap_{\mathcal{G}} I(d,G_{\alpha})$ in $F$. If $\mathcal{G}=\{ GL_n(R)\mid  R$ is admissible$\}$,
then $I(d,\mathcal{G} )$ is called the group of identities in $n \times n$ matrices, which is denoted $I(d ,n)$.
If $\mathcal{G} = \{I + M_n(m_R)\mid R$ is admissible$\},$ then we say that this is the group of
restricted identities for $n \times n$ matrices, denoted by $I^r(d, n).$ \label{d11}
\end{definition}

\begin{definition} \label{d12}
Let $d\geq2$ and $n$ be natural numbers. The prounipotent group $UG(n, d)$ of $d,$ $n\times n$ general matrices
is the closed subgroup in $I+M_n(m_S)$ generated by $X_p, 1 \leqslant p \leqslant d,$ where
$S = k[[x^{(p)}_{ij} | 1\leqslant i,j\leqslant n, 1\leqslant p\leqslant d]]$
is the ring of formal power series of $dn^2$ commuting variables, and $X_p=I+(x_{ij}^{(p)}).$
\end{definition}

There is a natural homomorphism $F(x_1, ... , x_d)\rightarrow UG(n, d)$ given on generators by the rule
$x_i\mapsto X_i$, whose kernel contains $I^r(d, n).$ In \cite{Mag} one proved the following Proposition, which is
an analog of earlier results of Amitsur (see for example \cite[Proposition 19]{For} and further references there).

\begin{proposition} \label{p18}
The natural homomorphism $F(x_1, ... , x_d)\rightarrow UG(n, d)$ induces the isomorphism of prounipotent groups
$F(x_1, ... , x_d)/I^r(d,n)\rightarrow UG(n, d).$
\end{proposition}

The following Theorem on non-triviality of identities will be needed below for constructing conjurings.
Recall that in free discrete groups there are no identities, they are linear.

\begin{proposition} \cite{Mag} \label{t6}
If $d,n \geqslant 2,$ then $I^r(d, n)\neq  \{e\}$.
\end{proposition}

\subsection{Proof of Theorem \ref{p02}}
\begin{proof} Since $GL_n(\mathbb{F}_p)$ is a finite group, the set
$M(n,d)$ of homomorphisms $\psi:F_p(d)\rightarrow GL_n(\mathbb{F}_p)$ is also finite.
Then $\mathbb{M}(n,d)=\cap_{\psi\in M(n,d)}ker\psi$ is a normal subgroup of finite index
in the free pro-$p$-group $F_p(d)$, and any element $g\in\mathbb{M}(n,d)$ is an identity, in the sense
that in any $n$-dimensional representation of the pro-$p$-group $F_p(d)$ over the field $\mathbb{F}_p$
the action of $g$ is trivial.

In free prounipotent groups there are identities (Proposition \ref{t6}). Also
$I^r(d,n)\lhd F_u(d):=F_u(d)(\mathbb{Q}_p)$ can be described (Proposition \ref{p18})
as the kernel of the homomorphism $\gamma_n$ onto $GU(n,d)$. Denote by $\widetilde{\gamma}_n$ the restriction
of $\gamma_n$
to the Zariski dense free pro-$p$-subgroup $F_p(d)\subseteq F_u(d)$.
$$\xymatrix{
  I^r(d,n)  \ar[r] & F_u(d)  \ar[r]^{\gamma_n} & GU(n,d) \\
  Z(d,n) \ar@{_{(}->}[u] \ar[r] & F_p(d) \ar@{_{(}->}[u] \ar[r]^{\widetilde{\gamma}_n} & GU^p(n,d) \ar@{_{(}->}[u] }$$
Assume the contrary, i.~e. that $ker(\widetilde{\gamma}_n)=1,$ then $\widetilde{\gamma}_n$ is an isomorphism.
It is clear that any central filtration $\widetilde{W}$ in $F_u$ induces a central filtration $W$ on $F_p(d).$
On the other hand, by \cite[Lemma 7.5]{HM2003} one has the isomorphism of graded quotients, where
$\widetilde{\widetilde{W}}=\gamma_n(\widetilde{W})$ is a central filtration in $GU(n,d)$:
$$Gr^{\widetilde{W}}_mF_u(d)\cong Gr^W_mF_p(d)\otimes_{\mathbb{Z}_p}\mathbb{Q}_p\cong Gr^{\widehat{W}}_mGU^p(n,d)
\otimes_{\mathbb{Z}_p}\mathbb{Q}_p\cong Gr^{\widetilde{\widetilde{W}}}_mGU(n,d).$$
Thus, $\gamma_n$ provides an isomorphism of graded quotients
$Gr^{\widetilde{W}}_mF_u(d)\cong Gr^{\widetilde{\widetilde{W}}}_mGU(n,d),$ hence $ker(\gamma_n)=1$ ($\cap \widetilde{W}_n=1$).
But this is not so by Proposition \ref{t6}
($I^r(d,n)\neq 1).$ Therefore $Z(d,n)\neq 1$ and since $\mathbb{M}(n,d)$ has a finite index in $F_p(d),$ then $\mathcal{Z}(d,n)=Z(d,n)\cap\mathbb{M}(n,d)\neq 1.$ Let us call nontrivial elements of $\mathcal{Z}(d,n)$
by \emph{conjurings}.

  Assume that $r=w^{p^l}$. Since $d\geq 2 $, for any $r$ one can choose an conjuring
  $z_n\in \mathcal{Z}(d,n)$, not lying in the centralizer of $w^{p^l}$ (since $I^r(d,n)\lhd F_u(d)$ and since
  $d\geq2$, it is not cyclic, hence non-Abelian, and therefore for any $r\in F_u$ there exists an conjuring
  $z_n \in Z(d,n),$  not lying in the centralizer of $r$ ($z_n \notin Z(r)$). Now the ``Fermat equality''
  $$w^{p^l}\cdot z_n^p= u^p$$ leads to a contradiction. Indeed, according to the pro-$p$-analog of the
  Lyndon--Sch\"{u}tzenberger theorem \cite[Theorem 1]{Mel2}, the rank of the free pro-$p$-subgroup generated by
  $\langle w^{p^{l-1}}, z_n, u\rangle$ equals one, and hence $z_n$ and $w^{p^n}$ commute,
  which contradicts to the choice of $z_n$, and therefore $w^{p^l}\cdot z_n^p$ is not a $p$-th power.

  Property b) follows by construction of $z_n$, since $z_n\in \mathcal{Z}(d,n)$.
\end{proof}

\begin{acknowledgement}
The author expresses gratitude to A.S. Mishchenko and V.M. Manuilov for the constant interest and valuable discussions during the work on the article.

\end{acknowledgement}

\end{document}